\documentclass[english]{article}

\usepackage[T1]{fontenc}
\usepackage{babel}
\usepackage{csquotes}
\usepackage{amsmath}
\usepackage{amsthm}
\usepackage{amssymb}
\usepackage{physics}
\usepackage{graphicx} 
\usepackage{enumitem}
\usepackage{xcolor}
\usepackage{listings}

\usepackage{biblatex}
\usepackage{hyperref}
\usepackage[capitalise]{cleveref}

\title{Explicit affine formulas for distances between tuples in classical discrete structures}
\author{Arthur Molina-Mounier}

\newcommand{\cL}{\mathcal{L}}
\newcommand{\cC}{\mathcal{C}}
\newcommand{\affset}[1]{\mathcal{L}^\text{aff}_{#1}}
\newcommand{\irng}[1]{[\![#1]\!]}

\DeclareMathOperator*{\argmin}{arg\,min}
\DeclareMathOperator*{\tp}{tp}

\newtheorem{thm}{Theorem}[section]
\newtheorem{prop}[thm]{Proposition}
\newtheorem{cor}[thm]{Corollary}
\newtheorem{lem}[thm]{Lemma}
\newtheorem{ques}{Question}[section]

\theoremstyle{definition}
\newtheorem{definition}[thm]{Definition}
\newtheorem*{notation}{Notation}

\theoremstyle{remark}
\newtheorem{fact}{Fait}[section]
\newtheorem{rem}[fact]{Remark}
\newtheorem{ex}[fact]{Example}

\begin{document}
\maketitle

\abstract{Answering a question of Ben Yaacov, Ibarlucía, and Tsankov \cite{EMDI}, we show an explicit way to construct an affine formula for the distance between two $n$-tuples in a $\{0,1\}$-valued $\varnothing$-structure, using $\lceil \log_2 n \rceil$ quantifier alternations.}

\tableofcontents

\newpage

\section{Introduction}

Originally introduced back in 1966 by Chang and Keisler \cite{CCC}, but more recently developed in a modern form and popularized by the work of Ben Yaacov et al. \cite{MTFMS}, continuous logic is a generalization of classical model theory that seeks to provide a formalism more suited to describing objects endowed with a metric structure. In continuous logic, all structures have a complete bounded metric space structure, with regards to which functions defined in the language are continuous, much like predicates and formulas, which are real-valued. Instead of the usual set of Boolean connectives, continuous logic allows the use of any continuous function.

Affine logic is a fragment of continuous logic, introduced in 2014 in works of Bagheri \cite{Bag1,Bag2}, in which connectives are restricted to just affine functions. Although it usually loses some expressivity, its model theory gains a rich convex structure, whose systemic study was initiated in \cite{EMDI}. In that same article, it was established through abstract methods that, for a certain class of continuous metric structures, every continuous formula can be approximated arbitrarily well by affine formulas. This property notably holds for the discrete classical structures with one sort (where, moreover, the approximation is exact). However, the method used gave no explicit description of the aforementioned affine formulas, even for some of the simplest formulas that use non-affine connectives.

\begin{prop}[\cite{EMDI}, 26.5, 26.7]
\label[prop]{prop:thetadef}
	Let $\mathcal{L} = \varnothing$ be the empty language (save for the metric $d$), $\ell \in \mathbb{N}^* \cup \qty{\infty}$, $M_\ell$ the unique countable classical $\mathcal{L}$-structure with $\ell$ elements, and $C_\ell$ its theory. Then, modulo $C_\ell$, every continuous formula is equal to an affine formula.

	In particular, for all $\ell, n \ge 2$, there exists $\theta_n \in \affset{2n}$ such that for all $M \models C_\ell$ and $\bar{a}, \bar{b} \in M^n$,
	\[\theta_n(\bar{a},\bar{b}) = \begin{cases}
		0 & \text{if } \bar{a} = \bar{b} \\
		1 & \text{otherwise}
	\end{cases}\]
\end{prop}

\begin{ques}[\cite{EMDI}, 26.8]
Is there a ``simple'' explicit expression for the affine formula $\theta_n$?
\end{ques}

We address this question, giving a partial answer by providing an effective construction of these formulas that is uniform in $\ell$. In the remainder of the text, we assume the reader is familiar with usual notions and notations of continuous logic. The main result of the paper is the following:

\begin{thm}
\label[thm]{thm:main}
	For all $n \ge 2$, there exists $\theta_n \in \affset{2n}$ with quantifier complexity (i.e. number of quantifier alternations) $\lceil \log_2 n \rceil$ such that, uniformly for every $\ell \ge 2$, every $\varnothing$-structure $M \models C_\ell$, and every $\bar{a},\bar{b} \in M^n$,
	\[\theta_n(\bar{a},\bar{b}) = d_n(a,b) = \begin{cases}
		0 & \text{if } \bar{a} = \bar{b} \\
		1 & \text{otherwise}
	\end{cases}\]
\end{thm}

\begin{rem}
	The case $\ell = 1$ is trivial ($\theta_n = 0$ for all $n$) and will not be discussed further.
\end{rem}

\subsection{Structure of the paper}
\label{sec:overview}

\Cref{sec:alg} presents a construction of the $\theta_n$ that proves \cref{thm:main}. However, the validity of this construction is only ensured with computer assistance; code that verifies the claims of the section is provided in \cref{app:code}.

A computer-assisted proof such as the one outlined there, while answering the original question, provides little insight into why the construction works, and the proof of its correctness is akin to a tedious case exhaustion. As an alternative, \cref{sec:elem} presents a more conventional and conceptual construction, with its correctness easier to grasp than the computer-assisted one. However, this comes at the cost of a slight weakening of \cref{thm:main}, formulated as \cref{thm:main_weak}.

\section{Notation}

In the following, for all $n \ge 1$ and any set $M$, $d_n$ denotes the distance function on $M^n$ defined by
\[\forall a, b \in M^n, d_n(a,b) = \begin{cases}0 & \text{if } a = b \\ 1 & \text{otherwise}\end{cases}\]
Furthermore, we will call $d_1$ the \emph{trivial metric} (on $M$).

Unless otherwise mentioned, we work in $M$ a $\varnothing$-structure with the trivial metric as only predicate, whose cardinality is $\ell$.

\section{Algorithmic construction}
\label{sec:alg}

\subsection{Theoretical basis}

The following result proves that is it enough to show $d_2$ is equal to an affine formula.

\begin{thm}
\label{thm:reduction}
Assume there is an explicit affine formula $\theta_2$ equal to $d_2$. Then there is an effective construction of affine formulas $\theta_n$ equal to $d_n$, for all $n > 2$.
\end{thm}

\begin{proof}
	We first show the result holds for all $n$ of the form $2^k$, $k \ge 1$.

	The property holds for $2 = 2^1$ by assumption. Let $k \ge 1$, and assume $\theta_{2^k}$ has been constructed. Then, by definition, it is the trivial metric $d'$ on $M^{2^k}$.

	Applying the hypothesis to this structure, we get that $\theta_2^{M^{2^k}}$ is the trivial metric on pairs of elements of $M^{2^k}$. By definition of $d'$, and by using the natural correspondence between $\qty(M^{2^k})^2$ and $M^{2^{k+1}}$, we can recursively replace every variable in the expression of $\theta_2^{M^{2^k}}$ by a $2^k$-tuple of variables, and all instances of $d'$ with an equivalent instance of $\theta_{2^k}$.

	This yields an expression of the trivial metric on $2^{k+1}$-tuples in terms of $d$, which we can define to be $\theta_{2^{k+1}}$.

	Now let $n > 2$ for which $\theta_n$ hasn't yet been constructed, and let $k \ge 1$ be the smallest such that $n \le 2^k$. By the above, $\theta_{2^k}$ has an effective construction. Then define
	\[\theta_n(a_1,\ldots,a_n,b_1,\ldots,b_n) = \theta_{2^k}(a_1,\ldots,a_n,\underbrace{a_1,\ldots,a_1}_{2^k-n},b_1,\ldots,b_n,\underbrace{a_1,\ldots,a_1}_{2^k-n})\]
\end{proof}

In a structure whose only predicate is the trivial metric, knowing the equality or non-equality of every pair of elements of an $n$-tuple is the same as knowing the type of the $n$-tuple (in the sense of model theory). Since there are finitely many pairs of coordinates, there are also finitely many types of $n$-tuples. This means that any assertion about a formula in such a structure can be checked by case exhaustion on the finite number of types of tuples.

The present construction relies on this very fact. As there are finitely many types of tuples of each length, the vector space of continuous functions (or even more generally, real-valued functions) on such types is finite-dimensional (this argument was also used in the original proof of \cref{prop:thetadef}), with dimension equal to the number of types of tuples (for quadruplets with $\ell \ge 4$, this number is $15$). Since every (affine) formula can be interpreted as one such function, the space of affine formulas up to pointwise equality has (at most) the same dimension. Therefore, if we find a generating set of this space, an expression for $\theta_2$ can be deduced by taking an appropriate linear combination of the basis elements.

\subsection{Algorithmic derivation}

Using the fact that formulas are completely characterized by their image of every type of tuple, we wrote a simple Python script to enumerate all types of quadruplets, and for any given formula in $\affset{4}$, compute the images of each type as a vector, which characterizes the formula up to equivalence. Then, the \texttt{numpy} package was used to compute the rank of these vectors. Proceeding by trial and error, we constructed the following set of formulas with rank 15, which is a basis when $\ell \ge 4$ since 15 is then the number of types of quadruplets.

Code is provided in \cref{app:code}.

\pagebreak
All of the following formulas have at most 4 free variables among $x,y,z,w$. Equivalent expressions are given for some formulas.
\begin{itemize}
\item $\varphi_1 = 1$
\item $\varphi_2 = d(x,y)$
\item $\varphi_3 = d(x,z)$
\item $\varphi_4 = d(x,w)$
\item $\varphi_5 = d(y,z)$
\item $\varphi_6 = d(y,w)$
\item $\varphi_7 = d(z,w)$
\item $\varphi_8 = \inf_a (1+d(x,a)+d(y,a)+d(z,a)) = \qty{x,y,z}$
\item $\varphi_9 = \inf_a (1+d(x,a)+d(y,a)+d(w,a)) = \qty{x,y,w}$
\item $\varphi_{10} = \inf_a (1+d(x,a)+d(z,a)+d(w,a)) = \qty{x,z,w}$
\item $\varphi_{11} = \inf_a (1+d(y,a)+d(z,a)+d(w,a)) = \qty{y,z,w}$
\item $\begin{aligned}\varphi_{12} &= \sup_a (4 - d(a,x) - d(a,y) - d(a,z) - d(a,w)) \\ & = \text{largest number of components of $(x,y,z,w)$ that are equal to one another}\end{aligned}$
\item $\varphi_{13} = \sup_a (d(a,x)+d(a,y)+d(a,z)-d(a,w)-2) = \begin{cases}1 & \text{if } w \in \qty{x,y,z} \\ 0 & \text{otherwise}\end{cases}$
\item $\varphi_{14} = \sup_a (d(a,x)+d(a,y)-d(a,z)-d(a,w)) = \begin{cases}0 & \text{if } \qty{a,b} = \qty{c,d} \\ 2 & \text{if } c = d \notin \qty{a,b} \\ 1 & \text{otherwise}\end{cases}$
\item $\varphi_{15} = \sup_a (2(d(a,x)-d(a,z)) + d(a,y)-d(a,w))$
\end{itemize}

\begin{ques}
Is there a ``nice'' interpretation of the 15th formula? Alternatively, can we find another basis of formulas which all have ``nice'' interpretations?
\end{ques}

\subsection{Proof of \cref{thm:main}}

Using the above basis, \texttt{numpy} was further used to obtain $\theta_2$ as a linear combination of elements of the generating set:
\begin{align}
\theta_2(x,y,z,w) &= \varphi_2 - \varphi_3 - \varphi_4 - \varphi_8 + \varphi_{10} + \varphi_{15} \nonumber \\
&\begin{aligned}
&= d(x,y) - d(x,z) - d(x,w) \\
& - \inf_a (1+d(x,a)+d(y,a)+d(z,a)) \\
& + \inf_a (1+d(x,a)+d(z,a)+d(w,a)) \\
& + \sup_a (2(d(a,x)-d(a,z)) + d(a,y)-d(a,w))
\end{aligned} \label{eq:qsum}
\end{align}
Bringing all quantifiers together into one expression, we obtain
\[\begin{split}
\theta_2(x,y,z,w) = &\sup_a \sup_b \inf_c \big[ d(x,y) - d(x,z) - d(x,w) \\
	& + 2d(a,x) - 2d(a,z) + d(a,y) - d(a,w) \\
	& -d(b,x) - d(b,y) - d(b,z) \\
	& + d(c,x) + d(c,z) + d(c,w) \big]
\end{split}\]

Although the generating set search presumed $\ell \ge 4$ and thus the full set of types of quadruplets, the formula was checked to hold for $\ell=2,3$ as well.

\begin{prop}
\label[prop]{prop:alg_quantalt}
The formulas $\theta_n$ obtained from the above definition of $\theta_2$ and the procedure from the proof of \cref{thm:reduction} can be written with $\lceil \log_2 n \rceil$ quantifier alternations.
\end{prop}

\begin{proof}
Given the proof of \cref{thm:reduction}, it suffices to show that the construction of $\theta_{2^{k+1}}$ from $\theta_{2^k}$, for all $k \ge 1$, only incurs 1 additional quantifier alternation. Let $k \ge 1$, and assume $\theta_{2^k}$ has been written in prenex form, i.e. quantifiers followed by a quantifier-free formula.

Applying the construction, every variable is replaced with a pair of variables, and every instance of $d$ with $\theta_2$. Rewriting $\theta_2$ using \cref{eq:qsum}, and pushing minus signs inside quantifiers, the inner quantifier-free formula becomes a sum of terms of three different forms:
\begin{itemize}
\item Quantifier-free terms
\item Terms of the form $\sup \psi$ with $\psi$ quantifier-free.
\item Terms of the form $\inf \psi$ with $\psi$ quantifier-free.
\end{itemize}

Since these quantified expressions are summed, we can interleave all of the quantifiers in any order while preserving the validity of the formula. By clustering quantifiers of the same type together, and making the outermost ones match the innermost quantifier of the surrounding prenex form, we add only 1 quantifier alternation to the amount we started with.

Since $\theta_2$ has 1 quantifier alternation, this shows that for all $k \ge 1$, $\theta_{2^k}$ has $k$ quantifier alternations. Then, for any $n$ not of the form $2^k$, $\theta_n$ has the same amount of quantifiers as $\theta_{2^{\lceil \log_2 n \rceil}}$, which proves the result, and also \cref{thm:main}.
\end{proof}

\section{Alternate elementary construction}
\label{sec:elem}

\subsection{Overview}

As announced earlier, this section contains an alternate construction of the $\theta_n$ of \cref{thm:main}, using conceptual arguments and without relying on computer assistance. However, this approach can only prove a slightly weaker version of the theorem:

\begin{thm}
\label[thm]{thm:main_weak}
	For all $n \ge 2$, there exists $\theta_n \in \affset{2n}$ with quantifier complexity $2\lceil \log_2 n \rceil - 1$ such that, uniformly for every $\boxed{\ell \ge 3}$, every $\varnothing$-structure $M \models C_\ell$, and every $\bar{a},\bar{b} \in M^n$,
	\[\theta_n(\bar{a},\bar{b}) = d_n(a,b) = \begin{cases}
		0 & \text{if } \bar{a} = \bar{b} \\
		1 & \text{otherwise}
	\end{cases}\]
\end{thm}

Points where the $\ell \ge 3$ assumption is required will be highlighted by boxing them like above. The reduction to the $n=2$ case still holds:

\begin{thm}
\label{thm:reduction_weak}
Assume there is an explicit affine formula $\theta_2$ equal to $d_2$ when $\ell \ge 3$. Then for all $n > 2$, there is an effective construction of affine formulas $\theta_n$ equal to $d_n$ when $\ell \ge 3$.
\end{thm}

\begin{proof}
	Analogous to the proof of \cref{thm:reduction_weak}. The only difference is ensuring that $\qty|M^{2^k}| \ge 3$, so that $\theta_2^{M^{2^k}}$ is indeed equal to the trivial metric on pairs. This is always true since $\ell > 1$, $k \ge 1$, so $\qty|M^{2^k}| \ge 2^{2^1} = 4 > 3$.
\end{proof}

\begin{rem}
\label[rem]{rem:x}
The main result on which this construction hinges is that the set $X  = \qty{(a,b,c,d) \in M^4 \mid a = c\ \vee\ b = d}$ satisfies a certain effective form of definability, from which we will be able to deduce an expression for $\theta_2$ in \cref{sec:proof}.
\end{rem}

\subsection{Constructibility in classical continuous theories}

\begin{definition}
Given a language $\mathcal{L}$, a metric $\mathcal{L}$-structure is \emph{classical} if the range of every predicate of the language (including distance) is contained in $\{0,1\}$. A continuous $\mathcal{L}$-theory is \emph{classical} if all of its models are classical. We shall write $\mathcal{C}_\mathcal{L}$ the class of classical $\mathcal{L}$-structures.
\end{definition}

\begin{definition}[Reminder : affinely definable set]
	Fix $n > 0$ and $\cL$ a language. Let $\cC$ a class of metric $\cL$-structures. A (uniformly) \emph{affinely definable set over $\cC$} is given by a family of nonempty sets $(D_M)_{M \in \cC}$ such that for every $M \in\cC$, $D_M \subseteq M^n$, and by $\psi$ an affine definable $\cL$-predicate over $\cC$, i.e. a uniform limit of affine $\cL$-formulas (uniform over $M \in\cC$), whose interpretation $\psi^M$ in every $M \in\cC$ is equal to $d(x,D_M) = \inf_{a \in D_M} d(x,a)$.
	
	Furthermore, $(D_M)_{M \in \cC}$ is an affinely definable set over $\cC$ if and only if, for every affine $\cL$-formula in one free variable $x$ and additional variables $y$, $\varphi(x,y)$, there exists an affine definable predicate equal to $\inf_{x \in D_M} \varphi^M(x,y)$.
\end{definition}

\begin{definition}[Constructible set]
	Fix $n > 0$. We will say a family of sets $D = (D_M)_{M \text{ an $\cL$-structure}}$ is (affinely) \emph{constructible} if there exists an affine formula $f \in \affset{n}$ such that
	\[\forall M \text{ an $\cL$-structure}, D_M = \argmin_{x \in M^n} f^M(x) = \qty{a \in M^n \mid f^M(a) = \inf_{x\in M^n} f^M(x)}\]
	
	In this case, we will write $D = \argmin_x f(x)$. The integer $n$ is the \emph{dimension} of $D$.
\end{definition}

\begin{ex}
	The definable set $\Omega$ such that $\Omega_M = M$ for every $\cL$-structure $M$ is constructible, as $\Omega = \argmin 0$.
\end{ex}

\begin{prop}
\label[prop]{prop:defconstr}
	Every constructible set over $\cC_\cL$ is definable on $\cC_\cL$.
\end{prop}

\begin{proof}
	By structural induction, in a classical metric structure, every formula takes on a finite set of values which does not depend on the choice of a particular structure, since the values taken on by predicates invariably lie in $\{0,1\}$. In particular, the finiteness of the range implies that all $\argmin$s within a classical structure are nonempty, i.e. all minimums are reached.

	Furthermore, from the discreteness of the range, there is $\varepsilon > 0$ such that, for all classical $\cL$-structures $M$ and $a \in M^n$,
	\[f(a) > \inf_{a' \in M^n} \implies f(a) \ge \varepsilon + \inf_{a' \in M^n} f(a')\]

	Let $\varphi \in \affset{n+m}$. By induction and using the bounds contained in $\cL$ (in particular, uniform over $\cC_\cL$), there are $r,s \in \mathbb{R}$ such that, uniformly, $r \le \varphi \le s$.

	Fix $b \in M^m$. Let $A = \argmin_x f(x)$. We claim that
	\begin{equation}
	\label{eq:constr_quant}
	\inf_{a \in A} \varphi(a,b) = \inf_a \qty(\qty(\frac{s-r}{\varepsilon}+1)\qty(f(a) - \inf_z f(z)) + \varphi(a,b))
	\end{equation}

	Let $M \in \cC_\cL$. As remarked above, $A_M \ne 0$, and for all $a \in A_M$, the expression within the $\inf$ on the right-hand side is equal to $0 + \varphi(a,b)$, which is less than or equal to $s$ by definition.

	By the definition of $\varepsilon$, if $a \notin A_M$,
	\[f(a) - \inf_z f(z) \ge \varepsilon\] and thus
    \begin{align*}
        \qty(\frac{s-r}{\varepsilon} + 1)\qty(f(a) - \inf_z f(z)) + \varphi(a,b) &\ge \qty(\frac{s-r}{\varepsilon} + 1)\varepsilon + \varphi(a,b) \\
        &\ge s-r + \varepsilon + r \\
        &= s + \varepsilon \\
        &> s
    \end{align*}

    The infimum can therefore only be reached if $a \in A$, hence the result.

    Analogously, we have
    \begin{equation}
    \label{eq:constr_sup}
    \sup_{a \in A} \varphi(a,b) = \sup_a \qty(-\qty(\frac{s-r}{\varepsilon}+1)\qty(f(a) - \inf_z f(z)) + \varphi(a,b))
    \end{equation}
\end{proof}

\begin{rem}
	It should be noted that constructibility could have alternatively been defined using definable predicates instead of mere formulas, and most of the general theorems would remain true, \textit{mutatis mutandis}. The converse of the above proposition would then be true, as the distance to a definable set is, by definition, a definable predicate.
\end{rem}

\begin{rem}
	For the sake of intuition, constructibility may be understood as an effective form of definability: if $A$ is a constructible set, provided we know (or can effectively construct) an explicit formula $f$ such that $A = \argmin f$, every quantification over $A$ can be explicitly rewritten in terms of $f$.
	
	In the rest of the text, we say a construction is ``effective'' if one would be able to derive an explicit expression for the result of the construction from the outlined steps.
\end{rem}

\begin{notation}
	The statement of \cref{thm:main} can be reformulated by introducing the theory $C_{\ge 3}$ of structures with at least 3 elements, and quantifying over all $M$ models of $C_{\ge3}$. Unless otherwise mentioned, we will now take $M$ to be a nondescript model of one of the theories $C_\ell$, $\ell \ge 3$ (equivalently: a model of $C_{\ge 3}$).
\end{notation}

\subsection{Types of tuples}
\label{sec:types}

\begin{definition}[Type labeling]
	Let $\mathcal{L} = \varnothing$. Let $a = (a_1, \ldots, a_n)$ an $n$-tuple. Its type $\tp(a)$ can be labeled by the following procedure:
    \begin{enumerate}
        \item Label all elements equal to $a_1$ by $\mathbf{0}$.
        \item While there remain unlabeled elements, take the first one (according to the ordering of the coordinates), and label all elements equal to it using the smallest natural number not already used as a label.
    \end{enumerate}

	This process outputs a finite sequence of natural numbers $(t_k)_{1 \le k \le n}$ such that $t_1 = 0$ and for all $k \in \irng{1,n-1}$, \[t_{k+1} \le \max_{1 \le i \le k} t_i + 1\] with the same type (relative to equality of natural numbers) as the starting tuple.

	Furthermore, every such sequence represents exactly one type, namely that of tuples such that elements at two coordinates are equal if and only if the labels at those coordinates are equal. This shows there is a bijection between the set of such sequences and types of $n$-tuples.\footnote{The number of such sequences (and therefore of types in a sufficiently large space) as a function of $n$, as well as alternate descriptions, is given in \cite{A000110}}
\end{definition}

\begin{ex}
	In $\mathbb{N}$, the type of $(5,3,5,6)$ would be represented by $\mathbf{0102}$, and that of $(5,6,4,4)$ by $\mathbf{0122}$.

	There are two types of pairs: $\mathbf{00}$, $\mathbf{01}$.

	There are 5 types of triples: $\mathbf{000}$, $\mathbf{001}$, $\mathbf{010}$, $\mathbf{011}$, $\mathbf{012}$.

	There are 15 types of quadruplets : $\mathbf{0000}$, $\mathbf{0001}$, $\mathbf{0010}$, $\mathbf{0011}$, $\mathbf{0012}$, $\mathbf{0100}$, $\mathbf{0101}$, $\mathbf{0102}$, $\mathbf{0110}$, $\mathbf{0111}$, $\mathbf{0112}$, $\mathbf{0120}$, $\mathbf{0121}$, $\mathbf{0122}$, $\mathbf{0123}$.
\end{ex}

\begin{rem}
	In a nonempty language $\cL \ne \varnothing$, this labeling describes the types inside the structure's empty-language reduct.
\end{rem}

Since the value of a formula only depends on the type of its parameter, for the purposes of quantification we may use the above labeling to enumerate the types of tuples inside a set, and we may conflate a given set with the set of types of its elements. For instance, if $A = \qty{(a,a) \mid a \in M}$, we may write $A =  \qty{\mathbf{00}}$.

\subsection{Properties of constructible sets}
\label{sec:csets}

\begin{rem}
	If, instead of considering the whole of $\cC_\cL$, we restricted ourselves to one given classical structure $M$, then $\varphi$ and $f$ (and by extension $A$) could be taken with parameters in $M$, without changing the result or the proof. In other words, if $b \in M$ is fixed, then $A = \argmin_x f(x,b) \subseteq M$ is definable in $M$.
\end{rem}

\begin{notation}
	If $A$ is a constructible set, $f_A$ denotes a formula such that $A = \argmin_x f_A(x)$, unless otherwise mentioned. Furthermore, if $A$ was obtained through an effective construction, $f_A$ will be taken to be the formula obtained by applying said construction.
\end{notation}

\begin{prop}[Stability properties of constructibility]
\label[prop]{prop:constr_stab}
	Fix $\mathcal{L}$ a language. We will work in a classical $\mathcal{L}$-structure $M$.
    \begin{enumerate}
        \item\label{item:prod} If $A, A'$ are constructible, then $A \times A'$ is constructible.
        \item\label{item:proj} If $A$ is constructible, every projection from $A$ onto a subset of its coordinates is constructible.
        \item\label{item:inter} If $A$ and $A'$ are constructible, not disjoint, and of equal dimension, then $A \cap A' = \argmin(f_A+f_{A'})$ is constructible.
        \item\label{item:restarg} If $B \subseteq M^n$ is constructible and $f \in \affset{n}$, then $A = \argmin_{x\in B} f(x) \subseteq M^n$ is constructible.
		\item\label{item:compr} If $A \subseteq M^n$ is constructible, $1 \le i_1 < \cdots < i_k \le n$ and $B \subseteq M^k$ constructible, then \[A' = \qty{(a_1, \ldots, a_n) \in A \mid (a_{i_1}, \ldots, a_{i_k}) \in B}\] is constructible, on the condition it is nonempty.
		\item\label{item:param} If $A \subseteq M^{n+m}$ is constructible, and if $a \in M^n$ satisfies $\exists a' \in M^m, (a,a') \in A$, then \[A' = \qty{a' \in M^m \mid (a,a') \in A}\] is constructible in $\mathcal{L}(a)$, and furthermore $f_{A'}(x') = f_A(a,x')$.
    \end{enumerate}
	
	\cref{item:compr} is a restricted form of comprehension that will be freely used in the following.
\end{prop}

\begin{proof}
	\begin{enumerate}
		\item If $A = \argmin_x f_A(x)$ and $A' = \argmin_{x'} f_{A'}(x')$, \[A \times A' = \argmin_{x,x'}\qty(f_A(x) + f_{A'}(x'))\]
		
		\item Assume that $A$ has dimension $n$, i.e. $A = \argmin_{a_1, \ldots, a_n} f_A(a_1, \ldots, a_n)$, and $1 \le i_1 < \cdots < i_k \le n$ for some $k \in \irng{1,n}$. Let $j_1 < \cdots <j_{n-k}$ be an increasing enumeration of $\irng{1,n} \setminus \qty{i_1, \ldots, i_k}$. Then
		\[\qty{(a_{i_1}, \ldots, a_{i_k}), a \in A} = \argmin_{a_{i_1}, \ldots, a_{i_k}} \inf_{a_{j_1}, \ldots, a_{j_{n-k}}} f_A(a_1, \ldots, a_n)\]
		\item Similar to \cref{item:prod}.
		
		\item Similar to \cref{prop:defconstr}. Let $r, s \in \mathbb{R}$ satisfying $\forall x \in M^n, r \le f_A(x) \le s$, and let $\varepsilon > 0$ such that $\forall x \in M^n, f_B(x) > 0 \implies f_B(x) \ge \inf_y f_B(y) + \varepsilon$. Then define
		\begin{align*}
			g(x) &= \qty(\frac{s-r}{\varepsilon} + 1)(f_B(x) - \inf_y f_B(y)) + f_A(x) \\
			A' &= \argmin_{x} g(x) \\
		\end{align*}
		
		If $x \in B$, i.e. $f_B(x) = \inf_y f_B(y)$, then
		\[g(x) = f_A(x) \le s\]
		Now assume $x \notin B$, i.e. $f_B(x) \ge \varepsilon$. Then
		\begin{align*}
			g(x) &\ge s-r + \varepsilon + f_A(x) \\
			&\ge s-r + \varepsilon + r \\
			&= s + \varepsilon \\
			&> s
		\end{align*}
		
		Therefore $\forall x \in B, \forall y \notin B, g(x) < g(y)$, and as a consequence
		\[A' = \argmin_{x} g(x) = \argmin_{x\in B} g(x) = \argmin_{x\in B} f_A(x)\]
		
		\item Let $\tilde{f}_B \in \affset{n}$ be a formula such that $\tilde{f}_B(a_1, \ldots, a_n) = f_B(a_{i_1}, \ldots, a_{i_k})$. Then
		\[A' = \argmin_{(a_1, \ldots, a_n) \in \argmin_x \tilde{f}_B(x)} f_A(a)\] which is constructible by the previous item. Note that even if $A'$ as defined in the statement is empty, the right-hand side of the above equation will nevertheless be well-defined, but will not describe the same set.
		
		\item We have \[A' = \argmin_{x'} f_A(a,x')\] 
		Indeed, $\inf_{x'} f_A(a,x') \ge \inf_{x,x'} f_A(x,x')$, and by assumption on $a$ we have equality, so every $a' \in \argmin_{x'} f_A(a,x')$ satisfies \[f_A(a,a') = \inf_{x'} f_A(a,x') = \inf_{x,x'} f_A(x,x')\] i.e. $(a,a') \in A$.
	\end{enumerate}
\end{proof}

\begin{prop}
\label[prop]{prop:pairdef}
	Let $\mathcal{L}$ be a language and $a,b \in M$. Then the set $\qty{a,b}$ is constructible in $\mathcal{L}(a,b)$.
\end{prop}

\begin{proof}\ \\
	Immediate by observing that \[\qty{a,b} = \argmin_x \qty(d(a,x) + d(b,x))\]
\end{proof}

\begin{prop}[Useful ``connectives'']
	The following sets are constructible:
	\begin{align}
		\Delta &= \qty{(a,b) \mid a = b} &= \qty{\mathbf{00}} \\
		\bar\Delta &= \qty{(a,b) \mid a \ne b} &= \qty{\mathbf{01}} \\
		U &= \qty{(a,b,c) \mid c = a\ \vee\ c = b} = \qty{(a,b,c) \mid c \in \qty{a,b}} &= \qty{\mathbf{000}, \mathbf{010}, \mathbf{011}} \label{eq_inpair}
	\end{align}
	
	Furthermore, for every $p,q,r,s \in M$, we have
	\begin{align}
		\inf_{x} \qty(d(p,x) + d(q,x) + d(r,x) + 1) &= \qty|\qty{p,q,r}| \label{eq:card3} \\
		\qty|\qty{p,q,r}| + \sup_x \qty(d(p,x) + d(q,x) + d(r,x) - d(s,x) - 2) &= \qty|\qty{p,q,r,s}| \label{eq:card4}
	\end{align}
\end{prop}

\begin{proof}\ \\
	For the first part of the proposition :
	\begin{itemize}
		\item $\Delta = \argmin_{a,b} d(a,b)$
		\item $\bar\Delta = \argmin_{a,b} -d(a,b)$
		\item $U = \argmin_{a,b,c} \qty(d(a,c) + d(b,c) - d(a,b))$
	\end{itemize}
	
	\noindent
	We proceed by cases for \cref{eq:card3} :
	\begin{itemize}
		\item If $p = q = r$, then $x = p = q = r$ gives the expected value of $1$.
		\item If $p = q \ne r$ (since the formula is symmetric in its free variables, we may consider any permutation thereof), then taking $x = p = q$ does result in the value of $2$.
		\item Otherwise, if $p, q, r$ are all distinct from one another, then the minimum is reached when $x$ is equal to any one of them, in which case the resulting value is $3$ as desired.
	\end{itemize}
	
	For \cref{eq:card4}, notice that the second term of the sum is equal to $1$ if $s \notin \qty{p,q,r}$, and $0$ otherwise: if $s\in\qty{p,q,r}$, without loss of generality $s = r$, then $d(r,x) - d(s,x) = 0$ for all $x$, and the maximum $0$ is attained for all $x \notin \qty{p,q}$ (which must exist since $\boxed{\ell \ge 3}$). Otherwise, $x = s$ yields the maximum $1$.
\end{proof}

\begin{cor}[One-dimensional union]
\label[cor]{cor:union1}
	If $A, B$ are constructible and have dimension 1, then $A \cup B$ is constructible.
\end{cor}

\begin{proof} 
	By \cref{prop:constr_stab} (\labelcref{item:prod,item:inter}), the set
	\[\qty(A \times B \times M)\cap U = \qty{(a,b,c) \mid a\in A, b\in B, c \in \qty{a,b}}\]
	is constructible, and its projection onto its third coordinate is $A \cup B$, which is then constructible by \cref{prop:constr_stab} (\labelcref{item:proj}).
\end{proof}

\begin{prop}[One-dimensional complement]
	If $A$ is constructible, different from $M$, and has dimension 1, then its complement (in $M$) $A^c$ is also constructible.
\end{prop}

\begin{proof}
	Notice that for all $x \in M$,
	\[\inf_{a \in A} d(x,a) = \begin{cases}0 & \text{if }x\in A \\ 1 & \text{otherwise}\end{cases}\]
	Therefore,
	\[A^c = \argmin_x -\inf_{a \in A} d(x,a)\]
	which is constructible because $A$ is.
	
	The assumption that $A \ne M$ is required to ensure that the previous equality does hold. If $A = M$, $\inf_{a \in A} d(x,a)$ is always equal to $0$, hence $\argmin_x -\inf_{a \in A} d(x,a) = M$ instead of $\varnothing$.
\end{proof}

\begin{rem}
	The $\qty{0,1}$-valuedness of $d$ is crucial in both of the previous propositions (specifically in the definition of $U$, in the case of \cref{cor:union1}), so these proofs will only generalize to higher dimensions once the $\theta_n$ are defined.
\end{rem}

\begin{notation}[Constructible comprehensions]
	In order to make notations less cumbersome, we will from now on freely write comprehensions where it is clear they can be defined constructively, e.g. $\qty{(a,b,c) \in X \mid a \ne b}$ for $\qty{(a,b,c) \in X \mid (a,b) \in\bar\Delta}$, which is constructible as soon as $X$ is by \cref{prop:constr_stab} (\labelcref{item:compr}).
	
	We will also use existential quantifiers in such comprehensions, as they can be interpreted using projections. For example,
	\[\qty{(a,b,c) \mid \exists d \ne c, d \in \qty{a,b}}\]
	should be interpreted as the projection of the set
	\[\qty{(a,b,c,d) \mid d \ne c\ \wedge\ d \in \qty{a,b}} = \qty{(a,b,c,d) \mid (c,d) \in \bar\Delta\ \wedge\ (a,b,d) \in U}\] on its first three coordinates, which is constructible by \cref{prop:constr_stab} (\labelcref{item:proj,item:compr}).
\end{notation}

\subsection{Proof of \cref{thm:main_weak}}
\label{sec:proof}

Now that the notion of constructible sets has been properly introduced, we can state the intermediate result that was mentioned in \cref{rem:x} :

\begin{prop}
\label{prop:x}
The set $X = \qty{(a,b,c,d) \in M^4 \mid a = c\ \vee\ b = d}$ is constructible.
\end{prop}

We will need to state and prove several intermediate lemmas before proving \cref{prop:x}.

\begin{lem}
\label[lemma]{lem:le2}
The set $C_{\le2} = \qty{(a,b,c) \mid \qty|\qty{a,b,c}| \le 2}$ is constructible.
\end{lem}

\begin{proof}
Consider the constructible set $E = \qty{(a,b,c,d) \mid c \ne d}$. For all $(a,b,c,d) \in E$, $c\ne d$ so $\qty|\qty{a,b,c,d}| \ge 2$, and this minimum is reached. Then
\[E' = \argmin_{(a,b,c,d) \in E} \qty|\qty{a,b,c,d}| = \qty{(a,b,c,d) \mid c \ne d\ \wedge\ \qty|\qty{a,b,c,d}| = 2}\]
is constructible, using \cref{eq:card4} and \cref{prop:constr_stab} (\labelcref{item:restarg}). We can then define $C_{\le2}$ as the projection of $E'$ onto its first three coordinates, constructible by \cref{prop:constr_stab} (\labelcref{item:proj}). More explicitly,
\[C_{\le2} = \qty{(a,b,c) \mid \exists d \ne c, \qty|\qty{a,b,c,d}| = 2}\]

Indeed, if $\qty|\qty{a,b,c}| = 1$, any $d \ne c$ works. If $\qty|\qty{a,b,c}| = 2$, there is an element of the set that is distinct from $c$, and $d$ can be taken equal to it. However, if $\qty|\qty{a,b,c}| = 3$, $\qty|\qty{a,b,c,d}| > 2$ will hold for any choice of $d$, showing the equality.
\end{proof}

\begin{lem}
\label[lemma]{lem:A1}
The set \[A_1 = \qty{(a,b,c,d) \mid \qty|\qty{a,b,c}| \le 2\ \vee\ b = d}\] is constructible.
\end{lem}

\begin{lem}
\label[lemma]{lem:A2}
The set \[A_2 = \qty{(a,b,c,d) \mid \qty|\qty{b,c,d}| \le 2\ \vee\ c = a}\] is constructible.
\end{lem}

\begin{proof}
The corollary follows from the first result by symmetry.

Let us show that
\[\qty{(a,b,c,d) \mid \qty|\qty{a,b,c}| \le 2\ \vee\ b = d} = \qty{(a,b,c,d) \mid \exists e, \qty|\qty{a,e,c}| \le 2\ \wedge\ (b = e\ \vee\ b = d)}\]

\boxed{\subseteq} If $\qty|\qty{a,b,c}| \le 2$, then $e = b$ works. Otherwise, any $e \in \qty{a,c}$ works.

\boxed{\supseteq} If $(a,b,c,d)$ is an element of the right-hand side, either $b = d$ and the result is clear, or $b = e$, meaning $\qty|\qty{a,b,c}| \le 2$.

Finally, the right-hand side of the equality is constructible by \cref{lem:le2,prop:pairdef}.
\end{proof}

\begin{rem}
	Consider the constructible set $A_1 \cap A_2$. Reasoning by case analysis shows that this set contains exactly the types \textbf{0000, 0001, 0010, 0011, 0100, 0101, 0102, 0110, 0111, 0112, 0121, 0122}. Intuitively, all the types that cannot be contained in $X$ (as in the statement of \cref{prop:x}) for cardinality reasons have been eliminated.
	
	Among the remaining ones, all of $\mathbf{0011, 0110, 0112, 0122}$ remain to be eliminated.
\end{rem}

\begin{lem}
\label[lemma]{lem:bconstr}
The set \[B = \qty{(a,b,c) \mid a = b = c\ \vee\ (a \ne b\ \wedge\ b \ne c)} = \mathbf{\qty{000, 010, 012}}\] is constructible.
\end{lem}

\begin{proof}
We claim that
\[B = \qty{(a,b,c)\mid \exists d \ne c, \exists e\ne a, b \in \qty{a,d}\ \wedge\ b \in \qty{c,e}}\]

\boxed{\subseteq} Let $(a,b,c) \in B$. If $a = b = c$, then any choice of $d$ and $e$ works. If $a \ne b$ and $b \ne c$, then $d = e = b$ works.
 
\boxed{\supseteq} Let $(a,b,c)$ be an element of the RHS, and let $d, e$ be as in the definition. If $b = a$, then since $a \ne e$ and $b \in \qty{c,e}$ we have $b = c$, whence $a = b = c$. Otherwise, if $b \ne a$, then $b = d \ne c$, i.e. $a \ne b$ and $b \ne c$.

Finally, the RHS is constructible by \cref{prop:pairdef}.
\end{proof}

\begin{lem}
\label[lemma]{lem:A3}
The set $A_3 = \qty{(a,b,c,d) \mid a = c\ \vee\ (b,c,d) \in B}$ is constructible.
\end{lem}

\begin{proof}
Let us show that
\[A_3 = \qty{(a,b,c,d) \mid \exists e, (b,e,d) \in B\ \wedge\ c \in \qty{a,e}}\]

\boxed{\subseteq} Let $(a,b,c,d) \in A_3$. If $a = c$, either $b = d$ and we can pick $e = b = d$, either $b \ne d$ and we pick $e$ different from both (using $\boxed{\ell \ge 3}$), such that in either case $(b,e,d) \in B$. Otherwise, if $(b,c,d) \in B$, taking $e = c$ works.

\boxed{\supseteq} Let $(a,b,c,d)$ be a member of the RHS, and take $e$ as in the definition. If $a = c$, there is nothing to do, so assume $a \ne c$. We must have $c = e$, heence $(b,e,d) = (b,c,d) \in B$.

Finally, since $B$ is constructible, $A_3$ also is.
\end{proof}

These seemingly disconnected results are actually a consequence of a broader theorem:

\begin{prop}[Restricted union]
Let $n > 1$, and $P, Q$ two constructible sets of dimension $n$. Suppose also that for all $x \in M^{n-1}$, there are $a, b \in M$ such that $(a,x) \in P$ and $(b,x) \in Q$. Then $P \cup Q$ is constructible.
\end{prop}

\begin{proof}
We claim that
\[P \cup Q = \qty{(a,x) \in M \times M^{n-1} \mid \exists b, \exists c, (b,x) \in P\ \wedge\ (c,x)\in Q\ \wedge\ a\in\qty{b,c}}\]
where the RHS is clearly constructible.

\boxed{\subseteq} Let $(a,x) \in P \cup Q$, with $a \in M$, $x \in M^{n-1}$. If $(a,x) \in P$, then it is enough to pick $b = a$ and any $c$ such that $(c,x) \in Q$ (which always exists by assumption). The case $(a,x) \in Q$ is analogous.

\boxed{\supseteq} If $(a,x)$ belongs to the RHS (with $a \in M$, $x \in M^{n-1}$), then there are $b, c$ such that $(b,x) \in P$, $(c,x) \in Q$ and $a = b$ or $a = c$. In one case like the other, we have $(a,x) \in P$ or $(a,x) \in Q$.
\end{proof}

We now have the required tools to prove \cref{prop:x}.

\begin{proof}[Proof of \cref{prop:x}]
We will show $X = A_1 \cap A_2 \cap A_3$, which is constructible by \cref{prop:constr_stab} (\labelcref{item:inter}) and \cref{lem:A1,lem:A2,lem:A3}.

\boxed{\subseteq} : Let $t = (a,b,c,d) \in X$. Assume $b = d$. The case $a=c$ is even simpler given the definition of $A_3$, so we may assume $a \ne c$. We then have $t \in A_1$ and since $\qty|\qty{b,c,d}| = \qty|\qty{b,c}| \le 2$, also $t \in A_2$. Since $a \ne c$ and $\qty|\qty{a,b,c}| \le 2$, either $b = a$ in which case $b \ne c \ne d$, either $b = c$ and $b = c = d$, and in any case $(b,c,d) \in B$, $t \in A_3$, and hence $t \in A_1 \cap A_2 \cap A_3$.

\boxed{\supseteq} : Let $t = (a,b,c,d) \in A_1 \cap A_2 \cap A_3$ and first assume $b \ne d$. Given that $t \in A_2$, if $a=c$ we are done, and else $\qty|\qty{b,c,d}| \le 2$. Then since $t \in A_3$, if $a=c$ we are also done, and else $(b,c,d) \in B$. Since $b \ne d$, it must be that $b \ne c \wedge c \ne d$, but in that case $\qty|\qty{b,c,d}| > 2 2$, a contradiction.

Now assume $a \ne c$. Since $t \in A_3$ we have $(b,c,d) \in B$. Since $t \in A_1$ we have $b = d$ (and we are done), or $\qty|\qty{a,b,c}| \le 2$, and therefore $a = b \ne c$ or $a \ne b = c$. In the former case, $(b,c,d) \in B$ implies $c \ne d$, but given that $t \in A_2$ and $a \ne c$, $\qty|\qty{b,c,d}| \le 2$ which forces $b = d$. In the latter case, $(b,c,d) \in B$ implies $b = c = d$, and in every case $b = d$.

In all cases, $t \in X$, and $X$ is constructible.
\end{proof}

\begin{cor}
\label[cor]{cor:theta2_def}
We may define $\theta_2$ in the following way :
\[\theta_2(a,b,c,d) = \sup_{\substack{x \in \qty{a,c} \\ y \in \qty{b,d} \\ (a,b,x,y) \in X}} \qty(d(a,x) + d(b,y))\]
\end{cor}

\begin{proof}
Let $f_X \in \affset{4}$ be the formula defining $X$, by constructibility. Then
\[X' = \qty{(x,y) \mid (a,b,x,y) \in X}\]
is constructible in $\mathcal{L}(a,b)$ by \cref{prop:constr_stab} (\labelcref{item:param}), since $(a,b,a,b) \in X$.

Therefore, the quantification is equivalent to one over $E = \qty(\qty{a,c} \times \qty{b,c}) \cap X'$ which is constructible over $\mathcal{L}(a,b,c,d)$ by \cref{prop:pairdef} and \cref{prop:constr_stab} (\labelcref{item:prod,item:inter}).

We must also check that this formula does indeed match the desired properties of $\theta_2$. If $(x,y) \in E$, by définition $x = a$ or $y = b$, hence $d(a,x) + d(b,y) \le 1$.

If $a = c$ and $b = d$, then $E = \qty{(a,b)}$ and \[\theta_2(a,b,c,d) = \sup_{(x,y) \in E} \qty(d(a,x) + d(b,y)) = d(a,a) + d(b,b) = 0\] Otherwise, if for example $a \ne c$ (the other case is analogous by symmetry), then $(c,b) \in E$ and the maximum of $1$ is reached, and $\theta_2(a,b,c,d) = 1$.
\end{proof}

\begin{cor}[General union]
	Let $A, B$ be constructible, of dimension $n$. Then $A \cup B$ is constructible.
\end{cor}

\begin{proof}
	Applying \cref{thm:reduction_weak}, we obtain an affine expression for $d_n$. In a similar manner to the 1-dimensional case, we define
	\[U_n = \qty{(a,b,c) \in M^{3n} \mid c \in \qty{a,b}} = \argmin_{a,b,c} \qty(d_n(a,c) + d_n(b,c) - d_n(a,b))\]
	and the remainder of the proof is identical to that of \cref{cor:union1}.
\end{proof}

The above proof strategy has relied on the assumption that $\ell \ge 3$, and taking $\ell = 2$ makes the proof fail at key points, in a way that does not seem easily repairable.

\begin{ques}
Is there an analogous construction to \cref{prop:x} in the case $\ell = 2$? If so, can it lead to a construction that is uniform over all $\ell \ge 2$?
\end{ques}

\subsection{Quantifier alternation analysis}

Thanks to \cref{cor:theta2_def,thm:reduction_weak}, we have obtained expressions for $\theta_n$ for all $n \ge 2$. To conclude the proof of \cref{thm:main_weak}, we must prove that they have quantifier complexity $2\lceil \log_2 n\rceil - 1$. To this end, we need to unravel the nested quantifications over constructible sets and check how the final expression behaves with regards to quantifier alternation.

Looking at \cref{eq:constr_sup} (in the proof of \cref{prop:defconstr}, page \pageref{prop:defconstr}), there are two possible sources of quantifiers (and thus extraneous quantifier alternation):
\begin{itemize}
\item The $\inf_z f(z)$ subexpression.
\item The quantified expression $\varphi$.
\end{itemize}

Notice that in this equation, the former is negated twice, meaning it creates, \emph{a priori}, an extra quantifier alternation with the outermost $\sup$, which is the quantifier used in the proof of \cref{cor:theta2_def}. Hence, it is important to check whether it can be rewritten in a way that does not create extraneous quantifier alternations.

Another potential source of quantifier alternations is \cref{prop:constr_stab} (\labelcref{item:proj}), which is often used implicitly throughout the proof, and whose defining formula makes use of $\inf$-quantification. Fortunately, this particular result can be sidestepped:

\begin{rem}
Suppose $A \subseteq M^2$ is constructible, and let $A' = \pi_1(A) = \qty{a \in M \mid \exists a', (a, a') \in A}$. Then for all formulas $\varphi$,
\[\inf_{a \in A'} \varphi(a) = \inf_{(a,a') \in A} \varphi(a)\]
This naturally generalizes to higher dimensions and more complex projections.
\end{rem}

To apply this result, we change definitions used in \cref{sec:proof} as follows:
\begin{align*}
C_{\le 2} &= \qty{(a,b,c,d) \mid c \ne d \wedge \qty|\qty{a,b,c,d}| = 2} = \argmin_{\qty{(a,b,c,d), c \ne d}} \qty|\qty{a,b,c,d}| \\
A_1 &= \qty{(a,b,c,d,e,f) \mid (a,e,c,f) \in C_{\le 2} \wedge b \in \qty{d,e}} \\
A_2 &= \qty{(a,b,c,d,e,f) \mid (b,e,d,f) \in C_{\le 2} \wedge c \in \qty{a,e}} \\
B &= \qty{(a,b,c,d,e) \in \qty{(a,b,c,d,e), a \ne e \wedge c \ne d} \mid b \in \qty{a,d} \wedge b \in \qty{c,e}} \\
A_3 &= \qty{(a,b,c,d,e,f,g) \mid (b,e,d,f,g) \in B \wedge c \in \qty{a,e}} \\
X &= \big\{(a,b,c,d,e_1,\ldots,e_7) \mid \\
	& \qquad(a,b,c,d,e_1,e_2) \in A_1 \\
	& \qquad\wedge (a,b,c,d,e_3,e_4) \in A_2 \\
	& \qquad\wedge (a,b,c,d,e_5,e_6,e_7) \in A_3\big\} \\
X' &= \qty{(x,y,e_1,\ldots,e_7) \mid (a,b,x,y,e_1,\ldots,e_7) \in X \wedge x \in \qty{a,c} \wedge y \in \qty{b,d}} \\
\theta_2(a,b,c,d) &= \sup_{(x,y,e_1,\ldots,e_7) \in X'} d(a,x) +d(b,y)
\end{align*}

In the following, we freely apply the results and definitions of \cref{sec:csets,sec:proof}.

\begin{align}
f_{X'}(x,y,e_1,\ldots,e_7) &= f_X(a,b,x,y,e_1,\ldots,e_7) \\ & +d(x,a) + d(x,c) \nonumber\\ &+ d(y,b) + d(y,d) \nonumber\\
f_X(a,b,x,y,e_1,\ldots,e_7) &= f_{A_1}(a,b,x,y,e_1,e_2) \\ &+ f_{A_2}(a,b,x,y,e_3,e_4) \nonumber \\ &+ f_{A_3}(a,b,x,y,e_5,e_6,e_7) \nonumber \\
f_{A_1}(a,b,x,y,e_1,e_2) &= f_{C\le 2}(a,x,e_1,e_2) + d(b,y) + d(b,e_1) \label{eq:a1} \\
f_{C_{\le 2}}(a,x,e_1,e_2) &= K_1 (-d(e_1,e_2) - \inf_{a',x',e_1',e_2'} -d(e_1',e_2')) \label{eq:cle2} \\ &+ \qty|\qty{a,e_1,x,e_2}| \nonumber\\
\qty|\qty{a,e_1,x,e_2}| &= \inf_z \qty(d(a,z) + d(e_1,z) + d(x,z)) \label{eq:c4} \\ &+ \sup_w \qty(d(w,a)+d(w,e_1)+d(w,x)-d(w,e_2)-2) \nonumber\\
f_{A_2}(a,b,x,y,e_3,e_4) &= f_{C\le 2}(b,y,e_3,e_4) + d(x,a) + d(x,e_3) \label{eq:a2} \\
f_{A_3}(a,b,x,y,e_5,e_6,e_7) &= f_B(b,e_5,y,e_6,e_7) + d(x,a) + d(x,e_5) \label{eq:a3}\\
f_B(b,e_5,y,e_6,e_7) &= K_2(2 -d(b,e_7)- d(y,e_6)) \label{eq:b}\\
  &+ d(e_5,b) + d(e_5,e_6)\nonumber\\
  &+ d(e_5,y) + d(e_5,e_7)\nonumber
\end{align}
and finally
\begin{equation}
\label{eq:theta2_alt}
\begin{aligned}
\theta_2(a,b,c,d) &= \sup_{(x,y,e_1,\ldots,e_7)}  \bigg[ \\
&-K_0\qty(f_{X'}(x,y,e_1,\ldots,e_7) - \inf_{x',y',e_1',\ldots,e_7'} f_{X'}(x',y',e_1',\ldots,e_7')) \\
& + d(a,x) + d(b,y) \bigg]
\end{aligned}
\end{equation}
where $K_0, K_1, K_2$ are real constants chosen as in \cref{eq:constr_quant}. Note the minus sign before $K_0$.

Noting that the $\inf$ in \cref{eq:theta2_alt} quantifies over every variable except $a,b$, we can make the following observations:
\begin{itemize}
\item The $\inf$ in \cref{eq:cle2} can be simplified to just $-1$.
\item Taking the $\inf$ over every variable except $a$ in \cref{eq:cle2} yields \[\inf_{a,b,c\ne d} \qty|\qty{a,b,c,d}| = 2\].
\item Doing the same in \cref{eq:a1,eq:a2} also yields $2$ for both, by picking $y=e_1=b$ (resp. $x=e_3=a$)
\item By the definition of $K_2$, the minimum can only be reached when $b \ne e_7$ and $y \ne e_6$ (which is possible since $\ell > 1$).

In such a case, $d(e_5,b) + d(e_5,e_7) \ge 1$ and $d(e_5,y) + d(e_5,e_6) \ge 1$, meaning the minimum is at least $2$. This is reached by picking $e_4=y=b$.
\item Similarly to above, the $\inf$ over every variable from $a,b$ in \cref{eq:a3} is $2$
\item Therefore
\[\inf_{x,y,e_1,\ldots,e_7} f_X(a,b,x,y,e_1,\ldots,e_7) = 2 + 2 + 2 = 6\]
and
\begin{align*}
\inf_{x,y,e_1,\ldots,e_7} f_{X'}(x,y,e_1,\ldots,e_7) &= 6 + \inf_x \qty(d(x,a) + d(x,c)) + \inf_y \qty(d(y,b) + d(y,d))\\
&= 6 + d(a,c) + d(b,d)
\end{align*}
\end{itemize}

Substituting this expression into \cref{eq:theta2_alt} and unfolding all definitions, we get the following result:

\begin{lem}
\label[lem]{lem:theta2_form}
The formula $\theta_2$ can be written as a $\sup$ quantifying over the sum of one $\inf$, one $\sup$ (both in \cref{eq:c4}), and many other quantifier-free terms. Hence $\theta_2$ can be written with only one quantifier alternation.
\end{lem}

In fact, the special form of the formula being quantified over allows us to state a strong result:

\begin{prop}
\label{prop:elem_quantalt}
The formulas $\theta_n$ obtained from the above definition of $\theta_2$ and the procedure from \cref{thm:reduction} can be written with $2\lceil \log_2 n \rceil - 1$ quantifier alternations.
\end{prop}

\begin{proof}
Analogous to the proof of \cref{prop:alg_quantalt}. The key difference lies in the analysis of the expression after substituting all instances of the metric with $\theta_2$.

Specifically, assume $k \ge 1$, and the expression of $\theta_{2^k}$ has been put in prenex form. Replacing variables with pairs of variables, instances of $d$ with $\theta_2$, and pushing minus signs inside quantifiers, the formerly quantifier-free inner formula becomes a sum of terms of three different forms:
\begin{itemize}
\item Quantifier-free terms
\item Terms of the form $\sup \qty(\sup \qty(\cdots) + \inf \qty (\cdots) + \cdots)$, corresponding to former instances of $d$ preceded by a plus sign.
\item Terms of the form $\inf \qty(\inf \qty(\cdots) + \sup \qty (\cdots) + \cdots)$, corresponding to former instances of $d$ preceded by a minus sign.
\end{itemize}
The latter two can be rewritten as formulas with one quantifier alternation. Since these quantified expressions are summed, we can interleave all of them while preserving the validity of the formula, as long as their relative orders do not change. It isn't too hard to see that any way of writing this prenex form will have at least two quantifier alternations. Doing this in such a way that the outermost quantifier of the sum matches the innermost quantifier of the initial prenex form, we add only 2 quantifier alternations to the amount we started with. Since $\theta_2$ can be written with just one quantifier alternation, this proves the result by induction.
\end{proof}

\appendix

\crefalias{section}{appendix}

\section{Code listing for \cref{sec:alg}}
\label{app:code}

\begin{lstlisting}[language=Python,escapeinside={\%}{\%}]
import numpy as np
from typing import Callable
import itertools as it

# Basic definition of affine formulas
class Var:
    counter = 0

    def __init__(self):
        self.key = Var.counter
        Var.counter += 1

    def __hash__(self):
        return self.key

class Valuation:
    def __init__(self, free_var_dict: dict[Var, int]):
        self.dict = free_var_dict
        self.max_val = max(free_var_dict.values())

    def __getitem__(self, var: Var):
        return self.dict[var]

    def __setitem__(self, var, val):
        if val > self.max_val + 1:
            raise AssertionError
        elif val == self.max_val + 1:
            self.max_val = val
        self.dict[var] = val
    def pop(self, var):
        self.dict.pop(var)
        self.max_val = max(self.dict.values())

class Formula:
    def __init__(self, vars_to_bind: set[Var]):
        self.to_bind = vars_to_bind

    def eval(self, valuation: Valuation) -> float:
        raise NotImplementedError

class Distance(Formula):
    def __init__(self, first: Var, second: Var):
        super().__init__({first, second})
        self.first = first
        self.second = second

    def eval(self, valuation) -> float:
        return 0 if valuation[self.first] == valuation[self.second] else 1

class AffineComb(Formula):
    def __init__(self, const: float, *comb: tuple[float, Formula]):
        super().__init__(set().union(*(x[1].to_bind for x in comb)))
        self.const = const
        self.comb = comb

    def eval(self, valuation) -> float:
        r = self.const
        for coeff, f in self.comb:
            r += coeff * f.eval(valuation)
        return r

def r():
    x = np.random.randint(-2, 3)
    print(x)
    return x

class Quantifier(Formula):
    def __init__(self, func: Callable[[float, float], float], form_constr: Callable[[Var], Formula]):
        fresh_var = Var()
        form = form_constr(fresh_var)
        bound = set(form.to_bind)
        bound.remove(fresh_var) # throws KeyError if not inside
        super().__init__(bound)
        self.func = func
        self.var = fresh_var
        self.formula = form

    def eval(self, valuation):
        valuation[self.var] = 0
        res = self.formula.eval(valuation)
        for vl in range(1, min(L, (valuation.max_val + 1) + 1)):
            valuation[self.var] = vl
            res = self.func(res, self.formula.eval(valuation))
        valuation.pop(self.var)
        return res

class InfQuantifier(Quantifier):
    def __init__(self, form_constr: Callable[[Var], Formula]):
        super().__init__(min, form_constr)

class SupQuantifier(Quantifier):
    def __init__(self, form_constr: Callable[[Var], Formula]):
        super().__init__(max, form_constr)

# Primitive for d_2
class D2(Formula):
    def __init__(self, *vars: Var):
        super().__init__(set(vars))
        self.vars = vars

    def eval(self, val: Valuation) -> float:
        (a, b, c, d) = self.vars
        return 0 if (val[a], val[b]) == (val[c], val[d]) else 1

# Generate all types of 'n'-tuples in a space of cardinality 'max_card'
def generate_types(n, max_card = None):
    if max_card is None:
        max_card = n + 1
    ts = [[0]]
    for _ in range(1, n):
        new_ts = []
        for t in ts:
            new_ts.extend(t + [i] for i in range(min(max(t) + 2, max_card)))
        ts = new_ts
    return ts

if __name__ == '__main__':
    n = 4
    # %$\ell$%, can be changed to 2 or 3
    L = 4
    free = [Var() for _ in range(n)]
    # Generate all types of tuples
    types = generate_types(n, max_card=L)
    # The %$\varphi_i$%, i from 1 through 15
    formulas: list[Formula] = [
        AffineComb(1.0),
        *[Distance(free[i], free[j]) for i in range(3) for j in range(i + 1, 4)],
        *[InfQuantifier(
            lambda a: AffineComb(
                1.0,
                (1, Distance(a, free[c[0]])),
                (1, Distance(a, free[c[1]])),
                (1, Distance(a, free[c[2]]))
            )
        ) for c in it.combinations(range(4), 3)],
        SupQuantifier(
            lambda a: AffineComb(
                4.0,
                (-1, Distance(a, free[0])),
                (-1, Distance(a, free[1])),
                (-1, Distance(a, free[2])),
                (-1, Distance(a, free[3])),
            )
        ),
        SupQuantifier(
            lambda a: AffineComb(
                -2.0,
                (1, Distance(a, free[0])),
                (1, Distance(a, free[1])),
                (1, Distance(a, free[2])),
                (-1, Distance(a, free[3])),
            )
        ),
        SupQuantifier(
            lambda a: AffineComb(
                0,
                (1, Distance(a, free[0])),
                (1, Distance(a, free[1])),
                (-1, Distance(a, free[2])),
                (-1, Distance(a, free[3]))
            )
        ),
        SupQuantifier(
            lambda a: AffineComb(
                0,
                (2, Distance(a, free[0])), # 2
                (1, Distance(a, free[1])),
                (-2, Distance(a, free[2])),
                (-1, Distance(a, free[3])) # -2
            )
        ),
        AffineComb(
            0.0,
            (1, Distance(free[0], free[1])),
            (-1, Distance(free[0], free[2])),
            (-1, Distance(free[0], free[3])),
            (-1, InfQuantifier(
                lambda a: AffineComb(
                    1.0,
                    (1, Distance(a, free[0])),
                    (1, Distance(a, free[1])),
                    (1, Distance(a, free[2]))
                )
            )),
            (1, InfQuantifier(
                lambda a: AffineComb(
                    1.0,
                    (1, Distance(a, free[0])),
                    (1, Distance(a, free[2])),
                    (1, Distance(a, free[3]))
                )
            )),
            (1, SupQuantifier(
                lambda a: AffineComb(
                    0,
                    (2, Distance(a, free[0])),
                    (-2, Distance(a, free[2])),
                    (1, Distance(a, free[1])),
                    (-1, Distance(a, free[3]))
                )
            ))
        )
    ]
    # Compute vector for each formula, and store as matrix
    A = np.ndarray((len(formulas), len(types)))
    for i, f in enumerate(formulas):
        for j, t in enumerate(types):
            val = Valuation({free[i]: t[i] for i in range(n)})
            A[i, j] = f.eval(val)
    # Check that the set is generating
    assert np.linalg.matrix_rank(A) >= len(types)

    theta_2 = SupQuantifier(
        lambda a: SupQuantifier(
            lambda b: InfQuantifier(
                lambda c : AffineComb(
                    0,
                    (1, Distance(free[0], free[1])),
                    (-1, Distance(free[0], free[2])),
                    (-1, Distance(free[0], free[3])),
                    (2, Distance(a, free[0])),
                    (-2, Distance(a, free[2])),
                    (1, Distance(a, free[1])),
                    (-1, Distance(a, free[3])),
                    (-1, Distance(b, free[0])),
                    (-1, Distance(b, free[1])),
                    (-1, Distance(b, free[2])),
                    (1, Distance(c, free[0])),
                    (1, Distance(c, free[2])),
                    (1, Distance(c, free[3]))
                )
            )
        )
    )

    # Check theta_2 is equal to d_2
    for t in types:
        d_2 = 0 if (t[0], t[1]) == (t[2], t[3]) else 1
        assert d_2 == theta_2.eval(Valuation({free[i]: t[i] for i in range(n)}))
\end{lstlisting}

\printbibliography
\end{document}